\documentclass[10pt]{article}

\usepackage{amsthm}
\usepackage[ansinew]{inputenc}   
\usepackage[english]{babel}
\usepackage{anysize}
\usepackage{amsfonts}
\usepackage{amssymb}
\usepackage{amsmath}
\usepackage{doc}
\usepackage{exscale}
\usepackage{fontenc}
\usepackage{ifthen}
\usepackage{latexsym}
\usepackage{makeidx}
\usepackage{syntonly}
\usepackage{graphicx}
\usepackage{float}
\usepackage{array}
\usepackage[all]{xy}
\usepackage[usenames,dvipsnames]{color}

\marginsize{2.5cm}{2.5cm}{2.5cm}{2.5cm}

\newtheorem{thm}{Theorem}[section]
\newtheorem{cor}[thm]{Corollary}
\newtheorem{lem}[thm]{Lemma}

\newtheorem{df}[thm]{Definition}
\newtheorem{rmk}[thm]{Remark}

\newcommand{\beq}{\begin{equation} }
\newcommand{\enq}{\end{equation}}

\newcommand{\ra}{\rightarrow}

\newcommand{\longra}{\longrightarrow}
\newcommand{\dra}{\dashrightarrow}

\newcommand{\mC}{\mathbb{C}}

\newcommand{\mG}{\mathbb{G}}

\newcommand{\mP}{\mathbb{P}}

\newcommand{\caO}{{\cal O}}

\DeclareMathOperator{\im}{{\rm im}}

\DeclareMathOperator{\Alb}{{\rm Alb}}

\DeclareMathOperator{\Sec}{{\rm Sec}}

\title{Grassmannian BGG Complex}
\author{}
\date{}

\begin{document}

\title{A generalization of the Castelnuovo-de Franchis inequality}

\author{V\'{\i}ctor Gonz\'{a}lez-Alonso
\thanks{The author completed this work supported by grant FPU-AP2008-01849 of the Spanish Ministerio de Educaci\'{o}n.} \\
\small{Departament de Matem\`atica Aplicada I, Universitat Polit\`ecnica de Catalunya (BarcelonaTech),}\\
\small{Av. Diagonal 647, 08028-Barcelona, Spain.} \\
\small{victor.gonzalez-alonso@upc.edu}
}

\maketitle

\begin{abstract}
In this article we give a lower bound on $h^{2,0}\left(X\right)$, where $X$ is an irregular compact K\"ahler (or smooth complex projective) variety, in terms of the minimal rank of an element in the kernel of $\psi_2: \bigwedge^2 H^0\left(X,\Omega_X^1\right) \ra H^0\left(X,\Omega_X^2\right)$. As a consequence, we obtain a generalization to higher dimensions of the Castelnuovo-de Franchis inequality for surfaces, improving some results of Lazarsfeld and Popa and Lombardi for threefolds and fourfolds.
\end{abstract}

\section{Introduction}

In the classification of higher dimensional algebraic varieties, a first step can be to decide whether the variety admits (or not) a fibration onto a variety of lower dimension. If the answer is positive, then one can reduce the problem to the study of the base and the fibres, which are of lower dimension and, somehow, eaiser than the original variety. Therefore, it is interesting to have any kind of criteria to decide the existence of fibrations whose total space is the given variety, and in particular, it is useful to know conditions on the numerical invariants of the variety (e.g. its Betti, Chern or Hodge numbers) implying that it is (or not) fibred.

A paradigmatical example is the classical Castelnuovo-de Franchis theorem, which says that an irregular surface $S$ admits a fibration onto a curve of genus $g \geq 2$ if and only if there are two holomorphic 1-forms whose wedge is zero. This theorem gives a numerical criterion in the spirit mentioned above: if the geometric genus $p_g\left(S\right)$ and the iregularity $q\left(S\right)$ of the surface satisfy
\beq \label{CdF-ineq-intro}
p_g\left(S\right) \leq 2q\left(S\right)-4,
\enq
then there exist two 1-forms wedging to zero, and therefore the variety is fibred.

The Castelnuovo-de Franchis theorem suggests that, for an irregular variety, its {\em higher irrational pencils} (fibrations analogous to surfaces fibred over curves of genus $g \geq 2$) are closely related to some special property of the algebra of holomorphic differential forms. Following this approach, Catanese \cite{Cat} and Ran \cite{Ran} proved independently a {\em Generalized Castelnuovo-de Franchis theorem}. As a consequence, one obtains that a non-fibred irregular variety $X$ must verify
\beq \label{GCdF-ineq-intro}
h^{k,0}\left(X\right) > k\left(q\left(X\right)-k\right)
\enq
for every $k = 1, \ldots, \dim X$.

A different approach is followed by Green and Lazarsfeld in \cite{GL1,GL2}, where they relate the same class of fibrations to the positive-dimensional components of some {\em cohomological support loci} of the variety. This alternative characterization led to a different generalization of the Castelnuovo-de Franchis inequality (\ref{CdF-ineq-intro}), obtained by Pareschi and Popa in \cite{PP}. The same inequality and some new ones were proven later by Lazarsfeld and Popa in \cite{LP}, using a completely different technique: the BGG complex. Using a similar construction (a BGG complex for the sheaves $\Omega_X^p$ of holomorphic $p$-forms), Lombardi obtained in \cite{Lom} more inequalities involving the Hodge numbers of varieties all whose 1-forms vanish at most at isolated points (a much more restrictive hypothesis than the non-existence of fibrations).

While the BGG complex takes into account only the multiplicative structure of the algebra $\oplus_{p=0}^d H^0\left(X,\Omega_X^p\right)$ of holomorphic forms, in this article we construct a generalization (the {\em Grassmannian BGG complex}) that also captures some of the additive structure. Unfortunately, we have not been able to relate the exactness of our complex to the existence of fibrations, as it was done in \cite{LP} for the BGG complex. However, we have been able to characterize the exactness of the shortest case in terms of the kernel of $\psi_2: \bigwedge^2 H^0\left(X,\Omega_X^1\right) \ra H^0\left(X,\Omega_X^2\right)$ and, as a byproduct, we can prove the following inequality, which generalizes (\ref{GCdF-ineq-intro}) for $k=2$ and improves some of the inequalities proven in \cite{LP} and \cite{Lom}.

\begin{thm}
Let $X$ be an irregular variety without higher irrational pencils. Then it holds
\beq \label{ineq-obj}
h^{2,0}\left(X\right) \geq
\begin{cases}
\binom{q\left(X\right)}{2} & \text{ if } q\left(X\right) \leq 2 \dim X -1 \\
2\left(\dim X-1\right)q\left(X\right) - \binom{2\dim X-1}{2} & \text otherwise
\end{cases}
\enq
\end{thm}

Although the case $q \leq 2d-1$ was already proven by Causin and Pirola in \cite{CP}, the bound in the general case is, as far as the author is aware, completely new (at least for dimension $d \geq 4$).

{\bf Acknowledgements:}

The author wants to thank the algebraic geometers and the PhD-students of the University of Pavia for his great hospitality during the stay while this work was developed, and specially Professor Gian Pietro Pirola for very inspiring conversations.


\section{Notations and preliminaries}

In this section we set the basic notation and recall the main known results which will be used along the article.

Throughout the paper, $X$ will denote a complex smooth irregular projective (or more generally, compact K\"ahler) variety of dimension $d = \dim X$. For the sake of brevity, we denote by $V = H^0\left(X,\Omega_X^1\right)$ the space of holomorphic 1-forms on $X$. The dimension of $V$ is $q = q\left(X\right)$, the {\em irregularity} of $X$, which is positive by assumption. We also denote by $h^{p,q} = h^{p,q}\left(X\right) = \dim_{\mC} H^q\left(X,\Omega_X^p\right)$ the Hodge numbers of $X$, as usual.

Let $A = \Alb\left(X\right)$ denote the Albanese torus of $X$, which is a $q$-dimensional complex torus (projective if $X$ is so), and $a=a_X: X \longra A$ the Albanese morphism of $X$.

\begin{df}[Maximal Albanese dimension, Albanese general type varieties]
An irregular variety $X$ is said to be of {\em maximal Albanese dimension} if $\dim a\left(X\right) = \dim X$ i.e., if the Albanese morphism is generically finite.

If furthermore $a$ is not surjective, i.e. $a\left(X\right) \subsetneq \Alb\left(X\right)$, $X$ is said to be of {\em Albanese general type}.

These definitions can be extended to non-smooth varieties considering any desingularization.
\end{df}

Equivalently, a variety is of Albanese general type if it is of maximal Albanese dimension and $q\left(X\right) > \dim X$. For example, every irregular curve (i.e. of genus $g \geq 1$) is of maximal Albanese dimension, because the Albanese map is nothing but the Abel-Jacobi map. Moreover, the curves of Albanese general type are exactly the curves of genus $g \geq 2$.

For any $n=1,\ldots,d$, let
$$\psi_n: \bigwedge^n H^0\left(X,\Omega_X^1\right) \ra H^0\left(X,\Omega_X^n\right)$$
be the map induced by wedge product. Since $a^*: H^0\left(A,\Omega_A^1\right) \ra H^0\left(X,\Omega_X^1\right)$ is an isomorphism and $H^0\left(A,\Omega_A^n\right) \cong \wedge^n H^0\left(A,\Omega_A^1\right)$, we can identify $\psi_n$ with the pull-back $a^*: H^0\left(A,\Omega_A^n\right) \ra H^0\left(X,\Omega_X^n\right)$ of $n$-forms by the Albanese morphism. Because of this interpretation, the maps $\psi_n$ are very related to the geometry of $X$, and in particular, the existence of decomposable elements in $\ker \psi_n$ has very strong consequences, as was shown independently by Ran and Catanese (see Theorem \ref{GCdF} below).

The case $n=2$ has been studied by  Causin and Pirola in \cite{CP}, proving in particular that $\psi_2$ is injective for $q \leq 2d - 1$, and also by Barja, Naranjo and Pirola in \cite{BNP}, where the authors focus on the consequences of the existence of elements of rank $2d$ (what they call {\em generalized lagrangian forms}) in the kernel of $\psi_2$. Our aim is to go further in the study of $\ker \psi_2$ in order to obtain new lower-bounds on $h^{2,0}$.

We will now introduce some basic notions on fibrations of irregular varieties, as well as the main characterization of some of them in terms of the algebra of holomorphic forms. Recall that a {\em fibration} is a surjective flat morphism $f: X \ra Y$ of varieties which has connected fibres. When dealing with irregular varieties, one can consider some special classes of fibrations:

\begin{df}[Irregular fibration, higher irrational pencil]
A fibration $f: X \longra Y$ is called {\em irregular} if $Y$ is irregular. If furthermore $Y$ is of Albanese general type, then $f$ is said to be a {\em higher irrational pencil} (on $X$).
\end{df}

Note that irregular fibrations (resp. higher irrational pencils) are higher-dimensional analogues to fibrations over non-rational curves (resp. curves of genus $g \geq 2$).

The existence of higher irrational pencils is closely related to the maps $\psi_n$, as the following theorem shows:

\begin{thm}[Generalized Castelnuovo-de Franchis (\cite{Cat}, Theorem 1.14 or \cite{Ran}, Proposition II.1)] \label{GCdF}
If $w_1,\ldots,w_n \in H^0\left(X,\Omega_X^1\right)$ are linearly independent 1-forms such that $\psi_n\left(w_1 \wedge \cdots \wedge w_n\right) = 0$, then there exists a higher irrational pencil $f: X \longra Y$ over a normal variety $Y$ of dimension $\dim Y < n$ and such that $w_i \in f^* H^0\left(Y,\Omega_Y^1\right)$.
\end{thm}

Because of this theorem, we will deal with linear subspaces of $V= H^0\left(X,\Omega_X^1\right)$, hence with Grassmannian varieties. For any positive integer $k$, we will denote by $\mG_k = Gr\left(k,V\right)$ the Grassmannian of $k$-dimensional subspaces of $V$. Recall that $\mG_k$ is naturally a subvariety of the projective space $\mP_k = \mP\left(\bigwedge^k V\right)$ via the Pl\"ucker embedding.

In general, if $E$ is any vector space and $e \in E$ is a non-zero vector, we denote by $\left[e\right]$ the corresponding point in $\mP\left(E\right)$. With this notation, the Pl\"ucker embedding maps the subspace spanned by $v_1,\ldots,v_k \in V$ to the point $\left[v_1\wedge\cdots\wedge v_k\right] \in \mP_k$.

We will also use symmetric powers of vector spaces and vector bundles. If $E$ is a vector space (or a vector bundle), we denote by $\Sigma^r E$ its $r$-th symmetric power, which is a quotient of $E^{\otimes r}$. We denote elements in $\Sigma^r E$ using multiplicative notation, so that if $e_1,\ldots,e_r \in E$ are arbitrary elements, we denote by $e_1\cdots e_r \in \Sigma^r E$ the image of $e_1 \otimes \cdots \otimes e_r$, and by $e_1^r$ the image of $e_ 1^{\otimes r} = e_1 \otimes \stackrel{r}{\cdots} \otimes e_1$.

Finally, we will also work with secant varieties of $\mG_k$ inside $\mP_k$. In general, if $X \subset \mP\left(E\right)$ is any projective variety, and $r$ is any positive integer, we denote by $\Sec^r\left(X\right) \subseteq \mP\left(E\right)$ the $r$-th secant variety of $X$ i.e., the closure of the union of the $\left(r-1\right)$-planes spanned by $r$ independent points in $X$. In particular, $\Sec^1\left(X\right) = X$ and $\Sec^2\left(X\right)$ is the usual secant variety of $X$. More explicitly, $\Sec^r\left(X\right)$ is the closure of the set
$$\left\{\left[e_1+\cdots+e_r\right] \, \vert \, \left[e_1\right],\ldots,\left[e_r\right] \in X\right\}.$$


\section{The Grassmannian BGG complex}

This section is devoted to explain the construction of our main tool, which we call the {\em Grassmannian BGG complex}. 

\begin{df}
Given two positive integers $r$ and $n \leq \min\{r,d\}$, and a linear subspace $W \subseteq V = H^0\left(X,\Omega^1_X\right)$, let $C_{r,n,W}$ be the complex
\beq \label{CrnW}
0 \ra \Sigma^rW \ra \Sigma^{r-1}W \otimes H^0\left(X,\Omega_X^1\right) \ra \cdots \ra \Sigma^{r-i}W \otimes H^0\left(X,\Omega_X^i\right) \ra \cdots \ra \Sigma^{r-n}W \otimes H^0\left(X,\Omega_X^n\right),
\enq
where the maps
$$\mu_i: \Sigma^{r-i}W \otimes H^0\left(X,\Omega_X^i\right) \longra \Sigma^{r-i-1}W \otimes H^0\left(X,\Omega_X^{i+1}\right)$$
are given by
$$\mu_i\left(\left(w_1 \cdots w_{r-i}\right)\otimes\alpha\right) = \sum_{j=1}^{r-i} \left(w_1 \cdots \hat{w_j} \cdots w_{r-i}\right)\otimes\left(w_j \wedge \alpha\right).$$
\end{df}

Since for every $1 \leq n' < n$ the complex $C_{r,n',W}$ is a truncation of $C_{r,n,W}$, we may assume that $n$ is the greatest possible, that is $n = \min\{r,d\}$, and denote the complex simply by $C_{r,W}$.

\begin{lem}
The maps $\mu_i$ are well defined and make $C_{r,W}$ a complex.
\end{lem}
\begin{proof}
It is an straightforward computation.
\end{proof}

Note that for a 1-dimensional $W$, generated by $w$, we have $\Sigma^r W \equiv \mC\langle w^r \rangle \cong \mC$, and $C_{d,\mC\langle w \rangle}$ is nothing but the complex
$$0 \longra H^0\left(X,\caO_X\right) \stackrel{\wedge w}{\longra} H^0\left(X,\Omega^1_X\right) \stackrel{\wedge w}{\longra} \ldots \stackrel{\wedge w}{\longra} H^0\left(X,\omega_X\right),$$
which is complex-conjugate to the {\em derivative complex} studied by Green and Lazarsfeld in \cite{GL1}.

Our first aim is to study the exactness of $C_{r,W}$. More precisely, we wish to obtain conditions on $W$ which guarantee that $C_{r,W}$ is exact in some (say $m$) of its first steps, (i.e., $C_{r,m,W}$ is exact), because this exactness will provide several inequalities between the Hodge numbers $h^{p,0}\left(X\right)$. Since we want to consider different subspaces $W$, we ``glue'' all the complexes (\ref{CrnW}) with fixed $k = \dim W$ as follows.

Let $\mG = \mG_k = Gr\left(k,V\right)$ be the Grassmannian of $k$-planes in $V$, and let $S \subseteq V \otimes \caO_{\mG}$ be the tautological subbundle, the vector bundle of rank $k$ whose fibre over a point $W \in \mG$ is precisely the subspace $W \subseteq V$.

\begin{df}[Grassmannian BGG complex]
For any $r \geq 1$, the $r$-th {\em Grassmannian BGG complex} (of rank $k$) of $X$ is the complex of vector bundles on $\mG_k$
\beq \label{BGG-complexes}
C_{r}: 0 \ra \Sigma^rS \ra \Sigma^{r-1}S \otimes H^0\left(X,\Omega_X^1\right) \ra \cdots \ra \Sigma^{r-i}S \otimes H^0\left(X,\Omega_X^i\right) \ra \cdots \ra \Sigma^{r-n}S \otimes H^0\left(X,\Omega_X^n\right)
\enq
where $n = \min\{r,d\}$ and over each point $W \in \mG_k$ it is given by (\ref{CrnW}).
\end{df} 

\begin{rmk}
If $k=1$, then $\mG = \mP = \mP\left(H^0\left(X,\Omega_X^1\right)\right)$, $S = \caO_{\mP}\left(-1\right)$ and $\Sigma^rS= \caO_{\mP}\left(-r\right)$. So taking $k=1$ and $r=d$, the above complex is (the complex-conjugate of) the BGG complex introduced by Lazarsfeld and Popa in \cite{LP}. In this way, the Grassmannian BGG complexes can be seen as generalizations (hence the name), with the new feature that they capture also the additive structure of the algebra of holomorphic differential forms of $X$.
\end{rmk}

The interest of studying these complexes is that, whenever they are exact at some point $W \in \mG$, they provide some inequalities involving the Hodge numbers $h^{p,0}\left(X\right) = h^0\left(X,\Omega_X^p\right)$. These inequalities are much stronger when the complex is exact {\em at every point}, so that the cokernel sheaves 
of the maps $\mu_i$ are vector bundles and a deeper study of them is feasible. For example, the proof of the higher-dimensional Castelnuovo-de Franchis inequality given by Lazarsfeld amd Popa in \cite{LP} is based on the fact that the BGG sheaf (the cokernel of the last map of $C_d$ with $k=1$) is an indecomposable vector bundle on $\mP^{q-1}$.

In this paper we deal with the case $r=2$, which is ``easy'' to study more or less by hand, and we obtain some inequalities for $h^{2,0}\left(X\right)$ in any dimension for non fibred irregular varieties, or more generally, for varieties which do not admit a generalized Lagrangian form of low rank. These bounds coincide with those obtained by Causin and Pirola in \cite{CP} for low irregularity, but are stronger than those known in higher dimension for high irregularity.

However, for bigger values of $r$, the approach followed in this article becomes too messy and a different tool to analyze the exactness of the complex $C_r$ is needed.


\section{Bounds on $h^{2,0}$}

In this section we consider the complex
\beq \label{C22}
C_2: 0 \longra \Sigma^2 S \longra S \otimes H^0\left(X,\Omega_X^1\right) \longra \caO_{\mG} \otimes H^0\left(X,\Omega_X^2\right)
\enq
over a Grasmannian $\mG = \mG_{2k} = \mG\left(2k,V\right)$ for some $1 \leq k \leq \frac{q}{2}$, and use it to obtain lower bounds on $h^{2,0}\left(X\right)$.

It turns out that the exactness of (\ref{C22}) at a general point is related to the existence of bivectors of small rank in the kernel of $\psi_2$. We start defining such notion.

\begin{df}
An element $v \in \bigwedge^2 V$ is said to have rank $2k$ if it can be written as
$$v=v_1 \wedge v_2 + \ldots v_{2k-1} \wedge v_{2k}$$
for some linearly independent elements $v_1, \ldots, v_{2k} \in V$.
\end{df}

\begin{rmk}
If we represent $v$ as an antisymmetric $q \times q$ matrix $A$ with respect to some fixed basis of $V$, then the rank of $v$ coincides with the rank of $A$ (which is always even). In particular, any element $v \in \bigwedge^2 V$ has rank at most $q$, and the elements of rank 2 are precisely the (non-zero) decomposable elements. More generally, the set of bivectors of rank at most $2m$ is the cone over $\Sec^m\left(\mG_2\right) \subseteq \mP\left(\bigwedge^2 V\right)$.
\end{rmk}

We present now our main result.

\begin{thm} \label{thm_c22}
Fix a positive integer $k \leq \frac{q}{2}$. If every non-zero element in $\ker \psi_2$ has rank bigger than $2k$, then the complex (\ref{C22}) on $\mG_{2k}$ is generically exact.
\end{thm}
\begin{proof}
By the previous remark, the hypothesis is equivalent to $\mP\left(\ker \psi_2\right) \cap \Sec^k\left(\mG_2\right) = \emptyset$. In this case, the rational map $\pi = \mP\left(\psi_2\right) : \mP\left(\bigwedge^2 V\right) \dra \mP\left(H^0\left(X,\Omega_X^2\right)\right)$ restricts to a {\em morphism}
$$\pi_k = \pi_{|\Sec^k\left(\mG_2\right)} : \Sec^k\left(\mG_2\right) \longra \mP\left(H^0\left(X,\Omega_X^2\right)\right)$$
which is finite onto its image. Indeed, if it is not the case, then there exists a curve $C \subseteq \Sec^k\left(\mG_2\right)$ such that $\pi\left(C\right) = p$ is just a point. Such a curve is thus contained in the linear space $\pi^{-1}\left(p\right)$, which contains $\mP\left(\ker \psi_2\right)$ as a hyperplane, and hence $C$ should intersect it, contradicting the fact that $\pi_k$ is defined everywhere in $\Sec^k\left(\mG_2\right)$.

Now suppose that the complex (\ref{C22}) is not exact at a point $W \in \mG_{2k}$, i.e. the complex of vector spaces
\beq \label{C22W}
C_2: 0 \longra \Sigma^2 W \stackrel{\mu_0}{\longra} W \otimes H^0\left(X,\Omega_X^1\right) \stackrel{\mu_1}{\longra} H^0\left(X,\Omega_X^2\right)
\enq
is not exact. Fix $\{w_1,\ldots,w_{2k}\}$ any base of $W$. Since $\mu_0\left(w_i w_j\right) = w_i \otimes w_j + w_j \otimes w_i$ for any $i,j$, and the elements $w_i \otimes w_j$ are linearly independent in $W \otimes H^0\left(X,\Omega_X^1\right)$, $\mu_0$ is clearly injective, identifying $\Sigma^2 W$ with the subspace of $W \otimes H^0\left(X,\Omega_X^1\right)$ spanned by $\{w_i \otimes w_i\}_{1 \leq i \leq k} \cup \{w_i \otimes w_j + w_j \otimes w_i\}_{1\leq i<j\leq k}$.

Therefore, the lack of exactness must come from the central term, that is, there exist $\alpha_1 \ldots \alpha_{2k} \in H^0\left(X,\Omega_X^1\right)$ such that $\sum_{i=1}^{2k} w_i \otimes \alpha_i \not\in \im\mu_0$ but
$$\mu_1\left(\sum_{i=1}^{2k} w_i \otimes \alpha_i\right) = \psi_2\left(\sum_{i=1}^{2k} w_i \wedge \alpha_i\right) = 0.$$
By substracting a suitable element from $\mu_0\left(\Sigma^2 W\right)$, we can assume furthermore that $\alpha_i \not\in \mC\langle w_1, \ldots, w_i \rangle$ for every $i$. In particular, we may assume that $\alpha_{2k} \not \in W$.

Consider now the curve $C \subseteq \Sec^k\left(\mG_2\right)$ parametrized by
$$\gamma\left(t\right) = [\left(w_1 - t \alpha_2\right) \wedge \left(w_2 + t \alpha _1\right) + \cdots + \left(w_{2k-1} - t \alpha_{2k}\right) \wedge \left(w_{2k} + t \alpha_{2k-1}\right)], \qquad t \in \mC.$$
Let $p = \gamma\left(0\right) = [w_1 \wedge w_2 + \ldots + w_{2k-1} \wedge w_{2k}]$. The tangent vector to $C$ at $p$ (to the branch of $C$ given by the image of a neighbourhood of $t=0$) is the class of
$$v=\sum_{i=1}^{2k} w_i \wedge \alpha_i$$
in $T_{\mP\left(\bigwedge^2 V\right),p} = \left(\bigwedge^2 V\right)/\mC\langle w_1 \wedge w_2 + \ldots + w_{2k-1} \wedge w_{2k} \rangle.$ Since at least $\alpha_{2k} \not \in W$, this class is clearly non zero. However, its image by the differential or $\pi_k$ is precisely the class of
$$\psi_2\left(\sum_{i=1}^{2k} w_i \wedge \alpha_i\right)=0$$
in $T_{\mP\left(H^0\left(X,\Omega_X^2\right)\right),\pi\left(p\right)} = H^0\left(X,\Omega_X^2\right)/\mC\langle \psi_2\left(w_1 \wedge w_2 + \ldots + w_{2k-1} \wedge w_{2k}\right) \rangle$, so $\pi_k$ is ramified at $p$. Since the general point of $\mP\left(\bigwedge^2 W\right)$ is of the form $[w_1 \wedge w_2 + \ldots + w_{2k-1} \wedge w_{2k}]$ for some basis of $W$, we see that $\pi_k$ ramifies at every point in $\mP\left(\bigwedge^2 W\right)$.

To finish the proof, note that $\Sec^k\left(\mG_2\right)$ is the union of all the $\mP\left(\bigwedge^2 W\right)$ as $W$ varies in $\mG_{2k}$, so if (\ref{C22}) were not exact for a general (and hence for any) $W \in \mG_{2k}$, then $\pi_k$ would be ramified all over $\Sec^k\left(\mG_2\right)$, contradicting the fact that it is finite.
\end{proof}

Now an easy dimension count gives our inequality:

\begin{cor} \label{cor1}
If there is no element of rank $2k \leq q$ in $\ker \psi_2$, then
$$h^{2,0}\left(X\right) \geq 2rq - \binom{2r+1}{2}$$
for all $1 \leq r \leq k$.
\end{cor}
\begin{proof}
By Theorem \ref{thm_c22}, for every $1 \leq r \leq k$, the complex (\ref{C22}) over any $\mG=\mG_{2r}$ is generically exact. Let $W \in \mG_{2r}$ be such that
$$0 \longra \Sigma^2 W \longra W \otimes H^0\left(X,\Omega_X^1\right) \longra H^0\left(X,\Omega_X^2\right)$$
is exact. The cokernel of the last map has dimension
$$\dim H^0\left(X,\Omega_X^2\right)-\dim\left(W \otimes H^0\left(X,\Omega_X^1\right)\right)+\dim\left(\Sigma^2 W\right) = h^{2,0}\left(X\right)- 2rq + \binom{2r+1}{2},$$
which must be non-negative, giving the desired inequality.
\end{proof}

\begin{rmk}
The case $k=1$ is the classical Castelnuovo-de Franchis inequality. The case $k=2$ has been already considered in \cite{BNP} and \cite{LP}, where the same inequality is obtained.
\end{rmk}

The existence of low-rank elements in the kernel of $\psi_2: \bigwedge^2 H^0\left(X,\Omega_X^1\right) \ra H^0\left(X,\Omega_X^2\right)$ can be related to the existence of higher irrational pencils on $X$, and this will give us a more geometric hypothesis to apply Corollary \ref{cor1}.

\begin{lem} \label{lem1}
If $v \in \ker \psi_2$ has rank $2k > 0$, $k < d$, then there exists a higher irrational pencil $f:X \ra Y$ with $\dim Y \leq k$.
\end{lem}
\begin{proof}
The proof relies on Theorem \ref{GCdF}. By this theorem, it suffices to find a decomposable element $v_1 \wedge \cdots \wedge v_{k+1}$ in the kernel of $\psi_{k+1}$. Writing $v=v_1 \wedge v_2 + \ldots v_{2k-1} \wedge v_{2k}$ with the $v_i$ linearly independent, it is immediate that the element $v_1 \wedge v_3 \wedge \ldots \wedge v_{2k-1} \wedge v_{2k}$, obtained by wedging $v$ with $v_1 \wedge v_3 \wedge \ldots \wedge v_{2k-3}$, maps to zero by $\psi_{k+1}$ because $\psi_2\left(v\right)=0$.
\end{proof}

We immediately obtain the next

\begin{cor} \label{cor2}
If $X$ does not admit any irrational pencil, then
$$h^{2,0}\left(X\right) \geq 2rq - \binom{2r+1}{2}$$
for all $1 \leq r \leq \min\left\{\frac{q}{2},\dim X-1\right\}$.
\end{cor}
\begin{proof}
Simply observe that Lemma \ref{lem1} allows us to apply Corollary \ref{cor1} for any $k \leq \dim X - 1$.
\end{proof}

\begin{rmk}
If $q \leq 2d-1$, then Corollary \ref{cor2} gives $h^{2,0}\left(X\right) \geq 2rq - \binom{2r+1}{2}$ for every $1 \leq r \leq \frac{q}{2}$. In this case, the right-hand-side of the inequality attains its maximum for $r = \left\lfloor\frac{q}{2}\right\rfloor$, giving
$$h^{2,0}\left(X\right) \geq \binom{q\left(X\right)}{2},$$
as was already obtained by Causin and Pirola in \cite{CP}.

However, if $q \geq 2d$, the maximum is attained for $r = d-1$, and we obtain
$$h^{2,0}\left(X\right) \geq 2\left(\dim X-1\right)q\left(X\right) - \binom{2\dim X-1}{2},$$
which coincides with the classical Castelnuovo-de Franchis inequality for surfaces without irrational pencils. Moreover, this result says that for fixed dimension and big irregularity, $h^{2,0}$ behaves asymptotically at least as $2\left(d-1\right)q$. For threefolds, this bound coincides with the one proven (with slightly more restrictive hypothesis) by Lombardi in \cite{Lom}, but improves his results in dimension four.
\end{rmk}

\bibliography{biblio}
\bibliographystyle{abbrv}

\end{document}